\documentclass[11pt]{article}
\usepackage{amsmath}
\usepackage{amssymb}
\usepackage{amsthm}
\usepackage{enumerate}
\usepackage{graphicx,color}
\newtheorem{theorem}{Theorem}[section]
\newtheorem{lemma}[theorem]{Lemma}

\newtheorem{proposition}[theorem]{Proposition}
\newtheorem{remark}[theorem]{Remark}

\newtheorem{conjecture}[theorem]{Conjecture}
\newtheorem{question}[theorem]{Question}

\newcommand{\pr}[1]{\operatorname{\mathbf{P}}\left(#1\right)}

\newcommand{\prcond}[2]{\operatorname{\mathbf{P}}\left(#1\;\middle\vert\;#2\right)}

\newcommand{\E}[1]{\operatorname{\mathbf{E}}\left[#1\right]}

\newcommand{\vol}[1]{\operatorname{vol}\left(#1\right)}

\newcommand{\critical}{\mathrm{c}}

\newcommand{\IGNORE}[1]{}

\begin{document}

\title{Perturbing the hexagonal circle packing: a percolation perspective}

\author{Itai Benjamini\thanks{Weizmann Institute of Science, Rehovot, Israel\ \ Email: \hbox{itai.benjamini@weizmann.ac.il}}
\and Alexandre Stauffer\thanks{Computer Science Division, University
        of California, Berkeley CA, USA\ \ Email:
        \hbox{stauffer@cs.berkeley.edu}.
        Supported by a Fulbright/CAPES scholarship and NSF grants CCF-0635153 and
        DMS-0528488.}} %Part of this work was done while the author was doing a summer internship at Microsoft Research, Redmond WA.}

%\date{}
\maketitle

\begin{abstract}
   We consider the hexagonal circle packing with radius $1/2$ and perturb it by letting the circles move as independent Brownian motions for time $t$.
   It is shown that, for large enough $t$,
   if $\Pi_t$ is the point process given by the center of the circles at time $t$, then, as $t\to\infty$, 
   the critical radius for circles centered at $\Pi_t$ to contain an infinite 
   component converges to that of continuum percolation 
   (which was shown---based on a Monte Carlo estimate---by Balister, Bollob\'as and Walters to be strictly bigger than $1/2$). 
   On the other hand, for small enough $t$, we show (using a Monte Carlo estimate for a fixed but high dimensional integral) 
   that the union of the circles contains an infinite connected component. 
   We discuss some extensions and open problems.
% \newline
% \newline
% \emph{Keywords and phrases.} Hexagonal circle packing, Brownian motion, continuum percolation.
% \newline
% MSC 2010 \emph{subject classifications.}
% Primary 82C43; %Time-dependent percolation
% Secondary 60G55, % Point processes
%           60D05, % Geometric probability and stochastic geometry
%           60J65, % Brownian motion
%           60K35, % Interacting random processes; statistical mechanics type models; percolation theory
%           82C21. % Dynamic continuum models (systems of particles, etc.)
\end{abstract}

%\newpage
%\setcounter{footnote}{0}
%\setcounter{page}{1}

%############################################################################################
%############################################################################################
%############################################################################################
\section{Introduction}
Let $\mathcal{T}$ be the triangular lattice with edge length $1$ and let $\Pi_0$ be the set of vertices of $\mathcal{T}$. 
We see $\Pi_0$ as a point process and, to avoid ambiguity, we use the term \emph{node} to refer to the points of $\Pi_0$. 
Now, for each node $u\in \Pi_0$,
we add a ball of radius $1/2$ centered at $u$, and set $R(\Pi_0)$ to be the region of $\mathbb{R}^2$ obtained by the 
union of these balls; more formally, 
$$
   R(\Pi_0) = \bigcup_{x\in\Pi_0} B(x,1/2),
$$
where $B(y,r)$ denotes the closed ball of radius $r$ centered at $y$.
In this way, $R(\Pi_0)$ is the so-called hexagonal circle packing of $\mathbb{R}^2$;  
refer to~\cite{Conway} for more information on packings.
Clearly, the region $R(\Pi_0)$ is a \emph{connected} subset of $\mathbb{R}^2$. 

Our goal is to analyze how this set evolves as we let the nodes of $\Pi_0$ move on $\mathbb{R}^2$ according to 
independent Brownian motions. For any $t>0$, let $\Pi_t$ be the point process obtained after the nodes have moved for time $t$.
More formally, for each node $u\in\Pi_0$,
let $(\zeta_u(t))_t$ be an independent Brownian motion on $\mathbb{R}^2$ starting at the origin, and set
$$
   \Pi_t = \bigcup_{u\in\Pi_0} (u + \zeta_u(t)).
$$
A natural question is whether there exists a phase transition on $t$ such that $R(\Pi_t)$ has an infinite component for small $t$ but has only finite components for large $t$. 

Intuitively, for sufficiently large time, one expects that $\Pi_t$ will look like a Poisson point process with intensity $2/\sqrt{3}$, which is the density of nodes in the triangular lattice.
Then, for sufficiently large $t$, $R(\Pi_t)$ will contain an infinite component almost surely only if 
$R(\Phi)$ contains an infinite component where $\Phi$ is a Poisson point process with intensity $\lambda=2/\sqrt{3}$.
In the literature, $R(\Phi)$ is referred to as the Boolean model. For this model, it is known that 
there exists a value $\lambda_\critical$ so that, 
if $\lambda<\lambda_\critical$, 
then all connected components of $R(\Phi)$ are finite almost surely~\cite{MR}.
On the other hand, if $\lambda>\lambda_\critical$, then $R(\Phi)$ contains
an infinite connected component. The value of $\lambda_\critical$ is currently unknown and depends on the radius of the balls in the
definition of the region $R$. When the balls have radius $1/2$, it is known that $\lambda_\critical$ satisfies
$0.52 \leq \lambda_\critical \leq 3.38$~\cite[Chapter~8]{BR}, but these bounds do not answer whether $2/\sqrt{3}$ is smaller or larger than $\lambda_\critical$. 
However, using a Monte Carlo analysis, Balister, Bollob\'as and Walters~\cite{BBW05} showed that, with $99.99\%$ confidence, $\lambda_\critical$ lies 
between $1.434$ and $1.438$, which are both larger than $2/\sqrt{3}$.
We then have the following theorem, whose proof we give in Section~\ref{sec:largetime}.

\begin{theorem}\label{thm:percolation}
   If $\lambda_\critical > 2/\sqrt{3}$, where $\lambda_\critical$ is the critical intensity for percolation of the Boolean model with balls of radius $1/2$, 
   then there exists a positive $t_0$ so that, for all $t>t_0$, $R(\Pi_t)$ contains no infinite component almost surely.
\end{theorem}

We now give a different statement of Theorem~\ref{thm:percolation}.
For each $t$, there exists a critical radius $r_\critical(t)$ so that, adding balls of radius $r>r_\critical(t)$
centered at the points of $\Pi_t$ gives that the union of these balls contains an infinite component almost surely. 
Then, the next theorem, whose proof we give in Section~\ref{sec:largetime}, follows from the proof of Theorem~\ref{thm:percolation} and a result of Sinclair and Stauffer~\cite{SS10}.

\begin{theorem}\label{thm:radius}
   As $t\to\infty$, we have that $r_\critical(t)$ converges to the critical radius for percolation of the Boolean model with intensity $2/\sqrt{3}$.
\end{theorem}

In order to prove Theorems~\ref{thm:percolation} and~\ref{thm:radius}, we devote most of Section~\ref{sec:largetime} to prove Theorem~\ref{thm:largetime} below, which 
we believe to be of independent interest. 
Consider a tessellation of $\mathbb{R}^2$ into regular hexagons of side length $\delta \sqrt{t}$, where $\delta$ is an arbitrarily small constant.
Let $I$ denote the set of points of $\mathbb{R}^2$ that are the centers of these hexagons. Then, for each $i\in I$, 
denote the hexagon with center at $i$ by $Q_i$ and
define a Bernoulli random variable $X_i$ with parameter $p$ independently of the 
other $X_j$, $j\neq i$. 
Define 
$$
   C(p,\delta) = \bigcup_{i\in I \colon X_i=1}Q_i.
$$
When $p>1/2$, which is the critical probability for site percolation on the triangular lattice,
we have that $C(p,\delta)$ contains a
unique infinite connected component.
We are now ready to state our main technical result, whose proof is given in Section~\ref{sec:largetime}.

\begin{theorem}\label{thm:largetime}
   There exists a universal constant $c>0$ and, for any $p\in(0,1)$ that can be arbitrarily close to $1$ and any arbitrarily small $\delta>0$, 
   there exists a $t_0>0$ such that, for all $t>t_0$, we can
   couple $\Pi_t$, $(X_i)_{i\in I}$ and a Poisson point process $\Phi$ of intensity $\frac{2}{\sqrt{3}}+c\sqrt{\delta}$, which is independent of $\Pi_0$, so that
   $$
      R(\Pi_t) \cap C(p,\delta) \subset R(\Phi) \cap C(p,\delta).
   $$
\end{theorem}

In words, Theorem~\ref{thm:largetime} establishes that $\Pi_t$ is \emph{stochastically dominated} by a Poisson point process inside $C(p,\delta)$. As we show later in Lemma~\ref{lem:notentirer2},
$\Pi_t$ cannot be stochastically dominated by a Poisson point process in the whole of $\mathbb{R}^2$.

We note that the opposite direction of Theorem~\ref{thm:largetime} was established by Sinclair and Stauffer~\cite[Proposition~4.1]{SS10}, who proved that, under some conditions on the initial location of the nodes, 
after moving as independent Brownian motions for time $t$, the nodes
\emph{stochastically dominate} a Poisson point process. 
The result of Sinclair and Stauffer has been used and refined in~\cite{PSSS11,S11}, and turned out to be a useful tool in the analysis of increasing events for models of mobile nodes, such as the 
so-called percolation time in~\cite{PSSS11} and detection time in~\cite{S11}. We expect that 
the ideas in our proof of Theorem~\ref{thm:largetime} can help in the analysis of \emph{decreasing} events, which have so far received less attention.

We remark that the proof of our Theorem~\ref{thm:largetime} requires a more delicate analysis than that of Sinclair and Stauffer. In their case, nodes that ended up moving 
atypically far away during the time $[0,t]$ could be simply disregarded as it is possible to show that the typical nodes already stochastically dominate a Poisson point process. 
In our setting, no node can be disregarded, regardless
of how atypical its motion turns out to be. In order to solve this problem, we first consider what we call well-behaved nodes, which among other things satisfy that 
their motion during $[0,t]$ is contained in some ball of radius $c\sqrt{t}$ for some large constant $c$
(we defer the complete definition of well-behaved nodes to Section~\ref{sec:largetime}). 
The definition of well-behaved nodes is carefully specified so that 
any given node is likely to be well-behaved and, in addition, it is possible to show that well-behaved nodes are 
stochastically dominated by a Poisson point process inside $C(p,\delta)$. For the remaining nodes, which comprise only a small density of nodes, 
we use a sprinkling argument to replace them already at time $0$ by a (non-homogeneous) Poisson point process of low intensity. Then, even though the motion of the nodes that are not well behaved is hard to control,
we use the fact that they are a Poisson point process at time $0$ 
to show that, at time $t$, they stochastically dominate a Poisson point process of low intensity. 

Now, for the case when the nodes of $\Pi_0$ move for only a small time $t$, we believe the following is true.
\begin{conjecture}\label{conj:smalltime}
   There exists a $t_0>0$ so that, for all $t<t_0$, $R(\Pi_t)$ contains an infinite component almost surely.
\end{conjecture}
We are able to establish the conjecture above given a Monte Carlo estimate for a finite but high dimensional integral. We discuss the details in Section~\ref{sec:smalltime}.
Note that if Conjecture~\ref{conj:smalltime} is true, then a curious consequence of this and Theorem~\ref{thm:radius} is that $r_\critical(t)$ is not monotone in $t$. 

We conclude in Section~\ref{sec:open} with some extensions and open problems.

%############################################################################################
%############################################################################################
%############################################################################################
\section{Stochastic Domination}\label{sec:stochdom}

We devote this section to the proof of our main technical result, Theorem~\ref{thm:largetime}, 
where we study the behavior of the balls after they have moved for a time $t$ that is sufficiently
large. In Section~\ref{sec:largetime} we show how to adapt this proof to establish Theorems~\ref{thm:percolation} and~\ref{thm:radius}.

Intuitively, as $t\to\infty$, $\Pi_t$ looks like an independent Poisson point process of intensity $2/\sqrt{3}$. 
Since the intensity of $\Phi$ is larger than $2/\sqrt{3}$, we would like to argue that there exists a coupling between $\Phi$ and $\Pi_t$ such that 
$\Phi$ contains $\Pi_t$. Unfortunately, this cannot be achieved in the whole of $\mathbb{R}^2$, as established by the lemma below, which gives that, for any fixed $t$, the probability that 
a sufficiently large region $S \subset \mathbb{R}^2$ contains no node of $\Pi_t$ is smaller than $\exp\left(-\left(\frac{2}{\sqrt{3}}+c\sqrt{\delta}\right)\vol{S}\right)$, which is the probability that $\Phi$ has no node in $S$.
Hence, $\Phi$ cannot stochastically dominate $\Pi_t$ in the whole of $\mathbb{R}^2$.
\begin{lemma}\label{lem:notentirer2}
   Fix $t$ sufficiently large and let $S$ be a hexagon of side length $(\log t)\sqrt{t}$ obtained as the union of $(\log t)^2 t$ triangles of $\mathcal{T}$. 
   Then, there exists a positive constant $c'$ such that
   $$
      \pr{\Pi_t \cap S = \emptyset} \leq \exp\left(-c'(\log t)^2 \vol{S}\right).
   $$
\end{lemma}
\begin{proof}
   For simplicity we assume that $(\log t)\sqrt{t}$ is an even integer.
   Let $x$ be the middle point of $S$ and 
   consider the hexagon $S'$ of side length $\frac{\log t}{2} \sqrt{t}$ composed of $\frac{(\log t)^2}{4} t$ triangles of $\mathcal{T}$ and centered at $x$. 
   Note that $S$ contains the ball $B\left(x,\frac{\sqrt{3}\log t}{2}\sqrt{t}\right)$ and $S'$ is contained in the ball $B\left(x,\frac{\log t}{2}\sqrt{t}\right)$.
   Therefore, a node of $\Pi_0$ that is inside $S'$ can only be outside of $S$ at time $t$ if it moves at least $\frac{(\sqrt{3}-1)\log t}{2}\sqrt{t} \geq \frac{\log t}{3}\sqrt{t}$.
   For any fixed node $u\in \Pi_0$, we have from the Gaussian tail bound (cf.\ Lemma~\ref{lem:gaussiantail}) that 
   $$
      \pr{\|\zeta_u(t)\|_2 \geq \frac{\log t}{3}\sqrt{t}} \leq \frac{3}{\sqrt{2\pi} \log t} \exp\left(-\frac{(\log t)^2}{18}\right).
   $$
   Each node of $\Pi_0$ belongs to $6$ triangles of $\mathcal{T}$, then there are at least $\frac{(\log t)^2}{24}t$ nodes of $\Pi_0$ in $S'$. Since each of them need to move more than $\frac{\log t}{3}\sqrt{t}$
   by time $t$ for $S$ to contain no node of $\Pi_t$, we obtain
   $$
      \pr{\Pi_t \cap S = \emptyset} 
      \leq \left(\frac{3}{\sqrt{2\pi} \log t} \exp\left(-\frac{(\log t)^2}{18}\right)\right)^{\frac{(\log t)^2}{24}t}
      \leq \exp\left(-\frac{(\log t)^4t}{432}\right).
   $$
   Since $\vol{S}=\frac{3\sqrt{3}}{2}(\log t)^2 t$, the proof is completed.
\end{proof}

Now we turn to the proof of Theorem~\ref{thm:largetime}.
The goal is to show that $\Phi$ contains $\Pi_t$ inside a percolating cluster of a suitable tessellation of $\mathbb{R}^2$.
For this, we tessellate $\mathbb{R}^2$ into hexagons of side length $\delta\sqrt{t}$. We take this tessellation in such a way that no point of $\Pi_0$ lies on the
edges of the hexagons; this is not crucial for the proof but simplifies the explanations in the sequel. Let $\mathcal{H}$ denote the set of these hexagons.
Consider a node $v\in \Pi_0$. Let $Q_i$ be the hexagon of $\mathcal{H}$ 
that contains $v$ and let $v'$ be a copy of $v$ located at the same position as $v$ at time $0$. 
We let $v'$ move up to time $t$ according to a certain procedure that we will describe in a moment, and then we say that $v$ is 
\emph{well behaved} if we are able to couple the motion of $v$ with the motion of $v'$ so that $v$ and $v'$ are at the same location
at time $t$. Recall that $I$ is the set of points given by the centers of the hexagons in $\mathcal{H}$.
For $i\in I$, we define
$$
   J_i = \Big\{ j\in I \colon \sup_{x\in Q_i, y\in Q_j}\|x-y\|_2 \leq C\delta \sqrt{t} \Big\},
$$
where
$$
   C=4\delta^{-3/2}.
$$
For $i,j$ such that $j\in J_i$ we say that $i$ and $j$ are neighbors.
Now we describe the motion of $v'$. Let $f_t$ be the density function for the location of a Brownian motion at time $t$ given that it starts at the origin
of $\mathbb{R}^2$. We fix $t$ and, for each $i,j\in I$ such that $i$ and $j$ are neighbors, we let 
\begin{equation} 
   \varphi_t(i,j) = \inf_{x\in Q_i,y\in Q_j} f_t(y-x).
   \label{eq:defvarphi}
\end{equation}
If $i$ and $j$ are not neighbors we set $\varphi_t(i,j)=0$.
Then, the motion of $v'$ is described by first choosing a $j\in J_{i}$ with probability proportional to 
$\varphi_t(i,j)$ and then placing $v'$ uniformly at random in $Q_j$. The main intuition behind this definition is that, when
$v$ is well behaved, its position inside $Q_j$ has the same distribution as that of a node of a Poisson point process inside $Q_j$. Therefore,
as long as the number of well behaved nodes that end up in $Q_j$ is smaller than the number of nodes in $\Phi\cap Q_j$, we will be able to couple them
with $\Phi$. Another important feature for the definition of well behaved nodes is that, if $v$ is well behaved and ends up moving to hexagon $Q_j$, then 
we know that, at time $0$, $v$ was in some hexagon of $J_j$. In particular, there is a bounded number of hexagons from which $v$ could have moved to $Q_j$, which 
allows us to control dependences.

Now we show that nodes are likely to be well behaved.
Since the area of each hexagon of $\mathcal{H}$ is 
$\frac{3\sqrt{3}}{2}\delta^2t$, we have that 
\begin{equation}
   \pr{\text{$v$ is well behaved}} = \sum_{j\in J_{i}} \frac{3\sqrt{3}}{2}\delta^2t \varphi_t(i,j).
   \label{eq:wellbehaved}
\end{equation}
The idea is that $\delta$ is sufficiently small so that $f_t$ varies very little (i.e., $f_t$ is essentially constant) 
inside any given hexagon of $\mathcal{H}$, but, at the same time, $C\delta$ is large so that the probability that 
$v$ moves to an hexagon that is not in $J_i$ is small.
We can then obtain in the lemma below that the probability that $v$ is well behaved is large.
\begin{lemma}\label{lem:behaved}
   Let $v$ be a node of $\Pi_0$ located in $Q_i$. We have
   $$
      (C-3)^2\leq |J_i| \leq \frac{4}{3}C^2,
   $$
   and, for sufficiently large $t$, we have
   $$
      \pr{\text{$v$ is well behaved}} \geq 1-5\delta.
   $$
\end{lemma}
\begin{proof}
   For $j\not\in J_i$, we know, by definition, that there exist a $x_0\in Q_i$ and a $y_0\in Q_j$ such that $\|x_0-y_0\|_2>C\delta\sqrt{t}$.
   Then, by the triangle inequality, we have that, for any $y\in Q_j$, 
   \begin{equation}
      \|y-i\|_2 \geq C\delta\sqrt{t}-\|y-y_0\|_2-\|i-x_0\|_2\geq C\delta\sqrt{t}-3\delta\sqrt{t},
      \label{eq:triang}
   \end{equation}
   where we used the fact that, for any two points $y,y_0$ in the same hexagon, we have $\|y-y_0\|_2 \leq 2\delta\sqrt{t}$ and, for any 
   $x\in Q_i$ we have $\|i-x\|_2\leq \delta\sqrt{t}$.
   Therefore, if we add balls of radius $\delta\sqrt{t}$ centered at each $j\in J_i$, 
   these balls cover the whole of $B(i,C\delta\sqrt{t}-3\delta\sqrt{t})$, which yields
   $$
      |J_i| \geq \frac{\vol{B(i,C\delta\sqrt{t}-3\delta\sqrt{t})}}{\vol{B(0,\delta\sqrt{t})}} 
      = \left(C-3\right)^2.
   $$
   For the other direction, note that if we add balls of radius $\frac{\sqrt{3}}{2}\delta\sqrt{t}$ centered at each $j\in J_i$, these balls are disjoint and 
   their union is contained in $B(i,C\delta\sqrt{t})$, which gives
   $$
      |J_i| \leq \frac{\vol{B(i,C\delta\sqrt{t})}}{\vol{B\left(0,\frac{\sqrt{3}}{2}\delta\sqrt{t}\right)}}
      = \frac{4}{3}C^2.
   $$
   
   Now we prove the second part of the lemma. 
   Note that, using~\eqref{eq:wellbehaved} and~\eqref{eq:defvarphi}, we have 
   \begin{eqnarray}
      \pr{\text{$v$ is well behaved}}
      &=& \sum_{j\in J_i} \frac{3\sqrt{3}}{2}\frac{\delta^2t}{2\pi t}\exp\left(-\frac{\sup_{x \in Q_i, y\in Q_j}\|x-y\|_2^2}{2t}\right)\nonumber\\
      &\geq& \sum_{j\in J_i} \int_{Q_j} \frac{1}{2\pi t}\exp\left(-\frac{(\|z-i\|_2+3\delta\sqrt{t})^2}{2t}\right)\,dz,
      \label{eq:boundphi}
   \end{eqnarray}
   where the last step follows by the triangle inequality.   
   Now, from~\eqref{eq:triang}, the ball $B(i,C\delta\sqrt{t}-3\delta\sqrt{t})$ only intersects hexagons that are neighbors of $i$. 
   We denote by $S_a=[-a/2,a/2]^2$ the square of side length $a$, and, 
   for any $z=(z_1,z_2)\in\mathbb{R}^2$ and $a\in \mathbb{R}_+$, we use the inequality $(\|z\|_2+a)^2\leq (|z_1|+a)^2+(|z_2|+a)^2$.
   Then, applying~\eqref{eq:boundphi}, we obtain
   \begin{eqnarray*}
      \pr{\text{$v$ is well behaved}}
      &\geq& \int_{B(0,C\delta\sqrt{t}-3\delta\sqrt{t})}\frac{1}{2\pi t}\exp\left(-\frac{(\|z\|_2+3\delta\sqrt{t})^2}{2t}\right)\,dz\\
      &\geq& \int_{S_{\frac{2C\delta\sqrt{t}-6\delta\sqrt{t}}{\sqrt{2}}}}\frac{1}{2\pi t}\exp\left(-\frac{(\|z\|_2+3\delta\sqrt{t})^2}{2t}\right)\,dz\\
      &\geq& \left(2\int_{3\delta\sqrt{t}}^{\frac{C\delta\sqrt{t}+3(\sqrt{2}-1)\delta\sqrt{t}}{\sqrt{2}}}\frac{1}{\sqrt{2\pi t}}\exp\left(-\frac{z_1^2}{2t}\right)\,dz_1\right)^2\\
      &\geq& \left(1-\frac{6\delta}{\sqrt{2\pi}} 
         - \frac{2}{\sqrt{\pi}(C+3(\sqrt{2}-1))\delta}\exp\left(-\frac{\delta^2(C+3(\sqrt{2}-1))^2}{4}\right)\right)^2\\
      &\geq& 1-\frac{12\delta}{\sqrt{2\pi}} 
         - \frac{4}{\sqrt{\pi}C\delta}\exp\left(-\frac{\delta^2C^2}{4}\right),
   \end{eqnarray*}
   where the second to last step follows by the standard Gaussian tail bound (cf.\ Lemma~\ref{lem:gaussiantail}). Then, using the value of $C$, we have that 
   $\frac{4}{\sqrt{\pi}C\delta}e^{-\delta^2C^2/4}\leq \frac{1}{\sqrt{\pi}}e^{-4}$ for all $\delta\in[0,1]$. 
   Then, the right-hand side above is at least $1-\delta\left(\frac{12}{\sqrt{2\pi}}+\frac{e^{-4}}{\sqrt{\pi}}\right)\geq 1-5\delta$.
\end{proof}

We will treat the nodes that are not well behaved by means of another point process. 
For any point $x\in\mathbb{R}^2$, we set $q(x)=i$ if $x\in Q_i$.
Then, let $g_t(x,y)$ be the density function for a node $v$ that is not well behaved to move from $x$ to $y$ after time $t$. We have that 
\begin{equation}
   g_t(x,y) = \frac{f_t(y-x)-\varphi_t(q(x),q(y))}{\pr{\text{$v$ is not well behaved}}}.
   \label{eq:defg}
\end{equation}
For each $v\in \Pi_0$, let $N_v(\mu)$ be a Poisson random variable with mean $\mu$, and let $\Psi_0(\mu)$ be the point process obtained by putting 
$N_v(\mu)$ points at $v$ for each $v\in\Pi_0$. 
We set $e^{-\mu}=\pr{\text{$v$ is well behaved}}$ and, from Lemma~\ref{lem:behaved} and the fact that $\delta$ is sufficiently small, we henceforth assume that 
$\mu\leq 1$. We can then use a standard coupling argument so that $N_v(\mu)\geq 1$ if and only if $v$ is not well behaved.
The intuition is that, by replacing each node of $\Pi_0$ that is not well behaved by a Poisson number of nodes, we can exploit the thinning property of 
Poisson random variables to show that, as the nodes move, they are stochastically dominated by a Poisson point process.

For each $w\in \Psi_0(\mu)$, let $\xi_w(t)$ be the position of $w$ at time $t$ according to the 
density function $g_t$.
Define $\Psi_t(\mu)$ to be the point process obtained by 
$$
   \Psi_t(\mu) = \bigcup_{w\in \Psi_0(\mu)}\xi_w(t).
$$
The following lemma gives that $\Psi_t(\mu)$ is stochastically dominated by a Poisson point process. 
\begin{lemma}\label{lem:problematic}
   Let $e^{-\mu}$ be the probability that a node of $\Pi_0$ is well behaved.
   For $t$ sufficiently large, 
   there exists a universal constant $c>0$ such that, if $\widetilde \Psi$ is a Poisson point process with intensity $c\sqrt{\delta}$,
   then it is possible to couple $\widetilde \Psi$ with $\Psi_t(\mu)$ so that $\Psi_t(\mu)\subseteq \widetilde \Psi$.
\end{lemma}
\begin{proof}
   Since the nodes of $\Phi_0(\mu)$ move independently of one another, we can apply the
   thinning property of Poisson random variables to obtain that $\Psi_t(\mu)$ is a Poisson point process. 
   Let $\Lambda(x)$ be the intensity
   of $\Psi_t(\mu)$ at $x\in \mathbb{R}^2$. 
   By symmetry of Brownian motion and the symmetry in the motion of well behaved nodes, we have that
   \begin{equation}
      \Lambda(x)
      = \sum_{v\in \Pi_0} \mu g_t(v,x)
      = \sum_{v\in \Pi_0} \mu g_t(x,v).
      \label{eq:sum}
   \end{equation}
   
   Recall that, for any $z\in \mathbb{R}^2$ and $\ell>0$, we define $z+S_{\ell}$ as the translation of the square $[0,\ell]^2$ so that its center is at $z$.
   Define the square $R_1$ as $x+S_{5\delta\sqrt{t}}$, the annulus $R_2$ as 
   $(x+S_{5C\delta\sqrt{t}})\setminus R_1$ and the region $R_3$ as 
   $\mathbb{R}^2\setminus (R_1\cup R_2)$.
   We split the sum in~\eqref{eq:sum} into three parts by considering the
   set of points $P_1 = \Pi_0\cap R_1$, 
   $P_2 = \Pi_0\cap R_2$ and 
   $P_3 = \Pi_0\cap R_3$.
   
   We start with $P_2$. 
   We can partition each hexagon of $\mathcal{H}$ into smaller hexagons of side length $\sqrt{3}/3$ such that each point of $\Pi_0$ is contained in exactly one such hexagon.
   This is possible since the dual lattice\footnote{Recall that the dual lattice of $\mathcal{T}$ is the lattice whose points are the faces of $\mathcal{T}$ 
   and two points are adjacent if their corresponding faces in $\mathcal{T}$ have a common edge.} 
   of $\mathcal{T}$ is a hexagonal lattice of side length $\sqrt{3}/3$, so the hexagons of side length $\sqrt{3}/3$ mentioned above can be obtained by translating and
   rotating the dual lattice of $\mathcal{T}$. 
   We denote by $\mathcal{H}'$ the set of hexagons of side length $\sqrt{3}/3$ obtained in this way.
   
   For each $z\in \mathbb{R}^2$, let $H_z$ be the hexagon that contains $z$ in $\mathcal{H}'$. 
   Each $H_z$ has side length $\sqrt{3}/3$ and area $\sqrt{3}/2$. 
   Thus, for any point $z\in R_2$, we have that $H_z\subset x+S_{5C\delta\sqrt{t}+4\sqrt{3}/3}$. 
   Let $R_2'= (x+S_{5C\delta\sqrt{t}+4\sqrt{3}/3}) \setminus R_1$, which gives that
   $$
      \sum_{v\in P_2} \mu g_t(x,v) \leq \frac{2}{\sqrt{3}} \int_{R_2'}\sup_{z'\in H_z}\mu g_t(x,z')\,dz.
   $$
   Now, note that $\frac{\mu}{\pr{\text{$v$ is not well behaved}}} = \frac{\mu}{1-e^{-\mu}} \leq \frac{1}{1-\mu/2}\leq 2$ since $\mu\leq 1$.
   Then, using the definition of $g_t$ from~\eqref{eq:defg} and the definition of $\varphi_t$ in~\eqref{eq:defvarphi},
   we have that 
   \begin{eqnarray*}
      \sum_{v\in P_2} \mu g_t(x,v)
      &\leq& \frac{4}{\sqrt{3}} \int_{R_2'}(\sup_{z'\in H_z}f_t(z'-x)-\varphi_t(q(x),q(z')))\,dz\\
      &\leq& \frac{4}{\sqrt{3}} \int_{R_2'}\left(\sup_{z'\in H_z}f_t(z'-x)-\inf_{x'\in Q_{q(x)},z''\in Q_{q(z')}}f_t(x'-z'')\right)\,dz.
   \end{eqnarray*}
   Now, by the triangle inequality, we have that $\|z'-x\|_2 \geq \|z-x\|_2-\|z-z'\|_2 \geq \|z-x\|_2-2\sqrt{3}/3$ and 
   $\|x'-z''\|_2\leq \|z-x\|_2+\|z-z''\|_2+\|x-x'\|_2 \leq \|z-x\|_2 + 4\delta\sqrt{t}$, which gives that 
   \begin{eqnarray*}
      \lefteqn{\sum_{v\in P_2} \mu g_t(x,v)}\\
      &\leq&
      \frac{4}{\sqrt{3}} \int_{R_2'}\frac{1}{2\pi t}
      \left(\exp\left(-\frac{(\|z-x\|_2-2\sqrt{3}/3)^2}{2t}\right)-\exp\left(-\frac{(\|z-x\|_2 + 4\delta\sqrt{t})^2}{2t}\right)\right)\,dz.
   \end{eqnarray*}
%    Note that, for any $w=(w_1,w_2)\in \mathbb{R}^2$ and constant $a\in\mathbb{R}_+$, we have that 
%    $(\|w\|_2-a)^2 \geq (|w_1|-a)^2+(|w_2|-a)^2-a^2$ and 
%    $(\|w\|_2+a)^2 \leq (|w_1|+a)^2+(|w_2|+a)^2$. This, and the fact that $R_2'$ is an annulus, gives that 
   Note that we can write 
   \begin{eqnarray*}
      \lefteqn{\exp\left(-\frac{(\|z-x\|_2-2\sqrt{3}/3)^2}{2t}\right)-\exp\left(-\frac{(\|z-x\|_2 + 4\delta\sqrt{t})^2}{2t}\right)}\\
      &=& \left(\exp\left(\frac{2\sqrt{3}\|z-x\|_2-2}{3t}\right)-\exp\left(-\frac{(4\|z-x\|_2\delta\sqrt{t} + 8\delta^2t)}{t}\right)\right)
      \exp\left(-\frac{\|z-x\|_2^2}{2t}\right).
   \end{eqnarray*}
   Now we use that, for $z\in R_2'$, we have $\|z-x\|_2 \leq \frac{5\sqrt{2}C}{2}\delta\sqrt{t}+2\sqrt{6}/3$.
   Then, the first exponential term above is $1+o(1)$ and, for the second exponential term, we can use the inequality $e^{-x}\geq 1-x$, which gives, as $t\to\infty$,
   \begin{eqnarray}
      \sum_{v\in P_2} \mu g_t(x,v)
      &\leq& \frac{4}{\sqrt{3}} (10\sqrt{2}C\delta^2+8\delta^2+o(1))\int_{R_2'}\frac{1}{2\pi t}\exp\left(-\frac{\|z-x\|_2^2}{2t}\right)\,dz\nonumber\\
      &\leq& c_1\sqrt{\delta} + o(1),
      \label{eq:p2}
   \end{eqnarray}
   for some universal constant $c_1>0$.
   
   For the terms of~\eqref{eq:sum} where $v\in P_3$ we have that $g_t(x,v)=\frac{f_t(x,v)}{\pr{\text{$v$ is not well behaved}}}$. Then,
   let $R_3'= \mathbb{R}^2 \setminus (x+S_{5C\delta\sqrt{t}-4\sqrt{3}/3})$ so that, for each $z\in R_3$, we have $H_z\subset R_3'$, which allows us to write
   $$
      \sum_{v\in P_3} \mu g_t(x,v) 
      \leq \frac{4}{\sqrt{3}} \int_{R_3'}\sup_{z'\in H_z}f_t(x,z')\,dz
      \leq \frac{4}{\sqrt{3}} \int_{R_3'}\frac{1}{2\pi t} \exp\left(-\frac{(\|z-x\|_2-2\sqrt{3}/3)^2}{2t}\right)\,dz.
   $$
   Now, letting $w=z-x$ and writing $w=(w_1,w_2)$ we have that 
   $$
      (\|w\|_2-2\sqrt{3}/3)^2 \geq (|w_1|-2\sqrt{3}/3)^2 + (|w_2|-2\sqrt{3}/3)^2 -4/3,
   $$
   which can be used to get the bound
   \begin{equation}
      \sum_{v\in P_3} \mu g_t(x,v) 
      \leq \frac{4}{\sqrt{3}} \exp\left(\frac{2}{3t}\right) 
         \left(2\int_{\frac{5C\delta\sqrt{t}}{2}-\frac{4\sqrt{3}}{3}}^\infty\frac{1}{\sqrt{2\pi t}} \exp\left(-\frac{w_1^2}{2t}\right)\,dw_1\right)^2
      \leq \frac{c_2}{C\delta}+o(1),
      \label{eq:p3}
   \end{equation}
   for some universal constant $c_2>0$.
   Finally, for the terms in~\eqref{eq:sum} with $v\in P_1$, we use that 
   $\mu g_t(v,x)\leq 2f_t(v,x) \leq \frac{1}{\pi t}$ for all $v,x$ which gives that 
   \begin{equation}
      \sum_{v\in P_1} \mu g_t(x,v) 
      \leq \frac{1}{\pi t} \frac{2}{\sqrt{3}} \left(5\delta\sqrt{t}+4\frac{\sqrt{3}}{3}\right)^2
      \leq c_3 \delta^2 + o(1),
      \label{eq:p1}
   \end{equation}
   for some universal constant $c_3>0$ and
   where $\frac{2}{\sqrt{3}} \left(5\delta\sqrt{t}+4\frac{\sqrt{3}}{3}\right)^2$ is an upper bound for the number of points in $P_1$.
   
   Plugging~\eqref{eq:p2},~\eqref{eq:p3} and~\eqref{eq:p1} into~\eqref{eq:sum} yields
   $$
      \Lambda(x) = \sum_{v\in \Pi_0} \mu g_t(x,v)
      \leq  c_1\sqrt{\delta} + \frac{c_2}{C\delta} + c_3\delta^2 + o(1).
   $$   
\end{proof}

We now proceed to the proof of Theorem~\ref{thm:largetime}.
\begin{proof}[{\bf Proof of Theorem~\ref{thm:largetime}}]
   We start by giving a high-level overview of the proof. 
   First, we assume that all nodes of $\Pi_0$ are well behaved. Then we consider a hexagon $Q_i$ of $\mathcal{H}$ and 
   count the number of such well behaved nodes that are inside $Q_i$ at time $t$. Note that, by the definition of well behaved nodes, given that
   a node is in $Q_i$ at time $t$, then its location is uniformly random in $Q_i$. Therefore, in order to show that they are stochastically dominated by
   a Poisson point process, it suffices to show that there are at most as many nodes of $\Pi_0$ in $Q_i$ at time $t$ as nodes of the Poisson point process. This will happen
   with a probability that can be made arbitrarily large by setting $t$ large enough. 
   We then use the fact that, since nodes are considered well behaved, a node 
   can only be in $Q_i$ at time $t$ if that node was inside a hexagon of $J_i$ at time $0$. Therefore, if we consider a hexagon $Q_j$ such that 
   $J_i \cap J_j=\emptyset$, we have that the well-behaved nodes that are able to be in $Q_i$ at time $t$ cannot end up in $Q_j$. Hence, the event 
   that the well-behaved nodes in $Q_i$ are stochastically dominated by a Poisson point process is independent of the event that the nodes in $Q_j$ 
   are stochastically dominated by
   a Poisson point process. This bounded dependency is enough to complete the analysis of well behaved nodes. On the other hand, to 
   handle nodes that are not well behaved, we add a discrete Poisson point process at each node of $\Pi_0$ so that the probability that we add at least one node 
   at a given $v\in\Pi_0$ is exactly the same as the probability that $v$ is not well behaved. Thus, this discrete Poisson point process contains
   the set of nodes that are not well behaved. We then use Lemma~\ref{lem:problematic} to conclude that 
   the nodes that are not well behaved are stochastically dominated by a Poisson point process, which concludes the proof.
   
   We now proceed to the rigorous argument.
   For each $v\in\Pi_0$, let $\xi'_v(t)$ be the position of $v$ at time $t$ given that $v$ is well behaved, and let 
   $$
      \Pi_t' = \bigcup_{v\in\Pi_0} \xi'_v(t).
   $$
   Note that, since $e^{-\mu}$ is the probability that a node is well behaved and $\Psi_0(\mu)$ is the point process obtained by adding a random number of 
   nodes to the points of $\Pi_0$ according to a Poisson random variable with mean $\mu$, then there exists a coupling so that 
   $$
      \Pi_t \subseteq \Pi_t' \cup \Psi_t(\mu).
   $$
   Lemma~\ref{lem:problematic} establishes that $\Psi_t(\mu)$ is stochastically dominated by a Poisson point process with intensity $c_1\sqrt{\delta}$ for 
   some universal constant $c_1>0$.
   It remains to show that $\Pi_t'$ is also stochastically dominated by a Poisson point process. Unfortunately, this is not true in the whole of 
   $\mathbb{R}^2$ as shown in Lemma~\ref{lem:notentirer2}. We will then consider the tessellation given by $\mathcal{H}$ and show that, for each hexagon $Q_i$ of the tessellation with
   $X_i=1$, where the $X_i$ are defined in the paragraph preceding Theorem~\ref{thm:largetime},
   $\Pi_t'$ is stochastically dominated by a Poisson point process $\widetilde \Pi$ of intensity $(1+\sqrt{\delta})2/\sqrt{3}$.
   
   In order to see this, for each $i\in I$, we define a binary random variable $Y_i$, which is $1$ if $\widetilde \Pi$ has more nodes in $Q_i$ than 
   $\Pi_t'$. Then, since each node of $\Pi_t'$ is well behaved, whenever $Y_i=1$, we can couple $\widetilde\Pi$ with $\Pi_t'$ such that 
   $\widetilde\Pi\supseteq \Pi_t'$ in $Q_i$. First we derive a bound for the number of nodes of $\Pi_t'$ inside $Q_i$. For each $v\in \Pi_0$, let 
   $Z_v$ be the indicator random variable for $\xi_v'(t)\in Q_i$. Then, since the probability that $\xi'_v(t)\in Q_i$ is proportional to $\varphi_t(q(v),i)$, the expected number of nodes of $\Pi_t'$ in $Q_i$ is 
   $$
      \sum_{v\in\Pi_0 \cap (\cup_{j\in J_i}Q_j)}\E{Z_v} 
      = \sum_{v\in\Pi_0 \cap (\cup_{j\in J_i}Q_j)} \frac{\varphi_t(q(v),i)}{M}
      = \sum_{v\in\Pi_0 \cap (\cup_{j\in J_i}Q_j)} \frac{\varphi_t(i,q(v))}{M}
      = 3\delta^2t,
   $$
   where $M$ is a normalizing constant so that $\sum_j \varphi_t(i,j)=M$ for all $i$.
   The last step follows since $\sum_{v\in\Pi_0 \cap (\cup_{j\in J_i}Q_j)} \varphi_t(i,q(v))=3\delta^2 t \sum_{j\in J_i} \varphi_t(i,j)=3\delta^2tM$.
   A simpler way to establish the equation above is by using stationarity and noting that 
   $3\delta^2t$ is the number of points of $\Pi_0$ in $Q_i$. 
   Since the random variables $Z_v$ are mutually independent, we can apply a Chernoff bound for 
   Binomial random variables (cf. Lemma~\ref{lem:cbbinomial}) to get
   $$
      \pr{\sum_{v\in\Pi_0 \cap (\cup_{j\in J_i}Q_j)} Z_v \geq (1+\sqrt{\delta}/2)3\delta^2t}
      %\pr{\text{$\Pi_t'$ has more than $(1+\beta)3\delta^2t$ nodes in $Q_i$}}
      \leq \exp\left(-\frac{2(\sqrt{\delta}/2)^2(3\delta^2t)^2}{3\delta^2t|J_i|}\right) 
      \leq \exp\left(-\frac{9\delta^3t}{8C^2}\right) ,
   $$
   where the last step follows from Lemma~\ref{lem:behaved}.
   Using a standard Chernoff bound for Poisson random variables (cf. Lemma~\ref{lem:cbpoisson}) we have
   $$
      \pr{\text{$\widetilde \Pi$ has less than $(1+\sqrt{\delta}/2)3\delta^2t$ nodes in $Q_i$}}
      \leq \exp\left(-\frac{\delta(3\delta^2t)}{2(1+\sqrt{\delta})}\right).
   $$
   Therefore, we obtain a constant $c_2$ such that 
   \begin{equation}
      \pr{Y_i=1} \geq 1-\exp\left(-\frac{c_2\delta^3t}{C^2}\right).
      \label{eq:yi}
   \end{equation}
   The random variables $Y$ are not mutually independent. However, note that $Y_i$ depends only on the random variables $Y_{i'}$ for which $J_{i'}\cap J_i\neq \emptyset$.
   This is because, for any $i\in I$, only the nodes that are inside hexagons $Q_j$ with $j\in J_i$ can contribute to 
   $Y_i$. Therefore, using Lemma~\ref{lem:behaved}, we have that $Y_i$ depends on at most $\left(\frac{4}{3}C^2\right)^2$ other random variables $Y$.
   By having $t$ large enough, we can make the bound in~\eqref{eq:yi} be arbitrarily close to $1$. This allows us to apply a result of 
   Liggett, Schonmann and Stacey~\cite[Theorem~1.3]{LSS97}, which gives that the random field $(Y_i)_{i\in I}$ stochastically dominates a field $(Y'_i)_{i\in I}$ of
   independent Bernoulli random variables satisfying
   $$
      \pr{Y_i'=1} \geq 1 - \exp\left(-\frac{c_3\delta^3t}{C^6} \right),
   $$
   for some positive constant $c_3$. So, with $t$ sufficiently large, we can assure 
   that $\pr{Y_i'=1}$ is larger than $p$ in the statement of Theorem~\ref{thm:largetime}.
   Then, we have that, whenever $Y_i'=1$, the Poisson point process $\widetilde \Pi \cup \Psi_t(\mu)$ stochastically dominates 
   $\Pi_t$ inside $Q_i$. Since $\widetilde \Pi$ and $\Psi_t(\mu)$ are independent Poisson point processes, we have that their union is 
   also a Poisson point process of intensity no larger than
   \begin{equation}
      \frac{2}{\sqrt{3}} +\frac{2}{\sqrt{3}}\sqrt{\delta} + c_1\sqrt{\delta},
      \label{eq:intensity}
   \end{equation}
   which completes the proof of Theorem~\ref{thm:largetime}.
\end{proof}

%############################################################################################
%############################################################################################
%############################################################################################
\section{Large time}\label{sec:largetime}
In this section we give the proofs of Theorems~\ref{thm:percolation} and~\ref{thm:radius}. 
Here we will use some steps of the proof of Theorem~\ref{thm:largetime} without repeating the details. For this reason, we suggest the reader to read the proof of Theorem~\ref{thm:largetime}
before embarking in the proofs in this section.

\begin{proof}[{\bf Proof of Theorem~\ref{thm:percolation}}]
   For each $i\in I$, let $N_i$ be the set of hexagons $Q_j$ such that $Q_i$ and $Q_j$ intersect. Now, define a binary random variable $\widetilde Y_i$ to be $1$ if $Y_j=1$ for all $j\in N_i$ and the largest component of 
   $R((\cup_{j\in N_i}Q_j) \cap (\widetilde \Pi \cup \Psi_t(\mu)))$ has diameter smaller than $\delta \sqrt{t}/10$. (Recall the definition of $Y_i$ from the paragraph preceding~\eqref{eq:yi}.) 
   If $\lambda_\critical>\frac{2}{\sqrt{3}}$, we can set $\delta$ small enough so that the intensity of $\widetilde \Pi \cup \Psi_t(\mu)$, which is bounded above by~\eqref{eq:intensity}, is smaller than $\lambda_\critical$. Then,
   using~\eqref{eq:yi} and a result of Penrose and Pisztora~\cite[Theorem~2]{PenPis}, 
   we have that, for any given $i$ and sufficiently large 
   $t$,
   \begin{equation}
      \pr{\widetilde Y_i=1} \geq 1 - 6\exp\left(-\frac{c_2\delta^3t}{C^2}\right) - \exp\left(-c_4\delta\sqrt{t} \right),
      \label{eq:ytilde}
   \end{equation}
   for some universal constant $c_4>0$. Also, the $\widetilde Y_i$ depends only on the random variables $Y_{i'}$ for which $J_{i'}\cap J_j\neq \emptyset$ for all $j\in N_i$. 
   Hence, $\widetilde Y_i$ depends on no more than $7\left(\frac{4}{3}C^2\right)^2$ other $\widetilde Y_j$ since $|N_i|=7$.
   Thus, as in the proof of Theorem~\ref{thm:largetime}, 
   we can apply the result of Liggett, Schonmann and Stacey~\cite[Theorem~1.3]{LSS97} to conclude that $(\widetilde Y_i)_i$ stochastically dominates a random field of independent Bernoulli random variables with mean $\widetilde p$,
   which can be made arbitrarily close to $1$ by having $t$ large enough. 
   As a consequence, we have that, almost surely, the set of hexagons with $\widetilde Y_i=0$ has only finite components.
   With this, we can conclude that $R(\Pi_t)$ has no infinite component almost surely since 
   the coupling between $\Pi_t$, $(\widetilde Y_i)_i$, $\widetilde \Pi$ and $\Psi_t(\mu)$ 
   gives that a connected component of $R(\Pi_t)$ can only intersect two non-adjacent edges of an hexagon $i\in I$ if $\widetilde Y_i=0$. But we showed that any set of intersecting hexagons with $\widetilde Y_i=0$ is finite and, therefore,
   must be surrounded by hexagons $j$ with $\widetilde Y_j=1$, which are not crossed by $R(\Pi_t)$. Hence, all components of $R(\Pi_t)$ are finite.
\end{proof}

\begin{proof}[{\bf Proof of Theorem~\ref{thm:radius}}]
   In this proof we will use some steps from the proof of Theorem~\ref{thm:percolation}.
   Let $r_\critical^\lambda$ be the critical radius for percolation of the Boolean model with intensity $\lambda$.
   Therefore, if $t$ is large enough and we add balls of radius $r<r_c^{\lambda_0}$ centered at the points of $\Pi_t$, where $\lambda_0$ is the intensity of $\widetilde \Pi\cup \Psi_t(\mu)$, which is bounded
   above by~\eqref{eq:intensity},
   we have that~\eqref{eq:ytilde} holds and the union of the balls do not have an infinite component.
   Now, we have that the Boolean model with intensity $\lambda$ and radius $r$ is equivalent (up to scaling) to the Boolean model with intensity $\bar \lambda$ and $\bar r$ provided 
   $\lambda r^2=\bar \lambda \bar r^2$. Therefore, for any $\epsilon>0$, we have that 
   \begin{equation}
      r_\critical^{\lambda + \epsilon} = r_\critical^\lambda \sqrt{\frac{\lambda}{\lambda+\epsilon}}.
      \label{eq:rscaling}
   \end{equation}
   Using this we obtain that, for any $\delta>0$, 
   $$
      \liminf_{t\to\infty} r_\critical(t) 
      \geq r_\critical^{\lambda_0}
      = r_\critical^{2/\sqrt{3}}\sqrt{\frac{2/\sqrt{3}}{2\sqrt{3}+c\sqrt{\delta}}},
   $$
   where $c>0$ is a universal constant.
   Since $\delta$ can be arbitrarily close to 0, we obtain 
   \begin{equation}
      \liminf_{t\to\infty} r_\critical(t)\geq r_\critical^{2/\sqrt{3}}.
      \label{eq:liminf}
   \end{equation}
   
   Now, to obtain an upper bound for $r_\critical(t)$, we use a result of Sinclair and Stauffer~\cite[Proposition~4.1]{SS10}, which can be stated as follows. 
   \begin{proposition}[{\cite[Proposition~4.1]{SS10}}]\label{pro:ss10}
      Consider 
      a square $S$ of side length $K$ tessellated into subsquares of side length $\ell$, and assume that each subsquare contains at least $\beta \ell^2$ nodes at time $0$. 
      Denote the nodes by $\Xi_0$ and let $\Xi_\Delta$ be the point process obtained after the
      nodes of $\Xi_0$ have moved as independent Brownian motions for time $\Delta$. Then, there are positive constants $c_1,c_2,c_3$ such that, for any $\epsilon>0$, 
      if $\Delta \geq c_1 \ell^2/\epsilon^2$, we can couple $\Xi_0$ with an independent Poisson point process of intensity $(1-\epsilon)\beta$ so that, 
      inside a square $S'$ of side length $K'\leq K-c_2\sqrt{\Delta \log \epsilon^{-1}}$ with the same center as $S$, 
      $\Xi_\Delta$ contains the nodes of the Poisson point process 
      with probability at least $1 - e^{-c_3 \epsilon^2\beta \ell^2 }$. 
   \end{proposition}
   Now we show how to apply this result to our setting. Let $i\in I$ be fixed. 
   Take $S$ to be the union of the hexagons in $J_i$; clearly $S$ is not a square but that is not important. Instead of tessellating $S$ into 
   squares, we tessellate $S$ into hexagons of side length $\sqrt{3}/3$. Each such hexagon contains one node of $\Xi_0=\Pi_0\cap (\cup_{j\in J_i} Q_j)$,
   which gives $\beta=2/\sqrt{3}$, the density of nodes of $\Xi_0$. 
   The main step in adapting the proof to our setting is to note that $\ell$ can represent the diameter of the cells of the tessellation, which in our case gives $\ell=2\sqrt{3}/3$, 
   and $K-K'$ is the minimum distance between $S'$ and a point outside $S$. Since we take $S'=Q_i$ we have that $K-K' \geq C\delta\sqrt{t}-2\delta\sqrt{t}$, where $2\delta\sqrt{t}$ is the 
   diameter of $Q_i$. Now setting $\epsilon = \sqrt{\delta}$ and $\Delta=t$, we obtain a $t_0$ so that, for all $t\geq t_0$, the conditions on $\Delta$ and $K-K'$ in Proposition~\ref{pro:ss10} are 
   satisfied. Therefore, with probability at least $1-e^{-c_3 \delta^3 t}$, for some positive constant $c_3$, the nodes of $\Pi_t$ that are inside $Q_i$ at time $t$ 
   and were inside $\cup_{j\in J_i} Q_j$ at time $0$ stochastically dominate a Poisson
   point process $\Phi$ of intensity $(1-\sqrt{\delta})\frac{2}{\sqrt{3}}$. 
   When this happens, we let $Y_i=1$, where $Y_i$ here is analogous to the one in~\eqref{eq:yi}. 
   
   Then we proceed similarly as in the proof of Theorem~\ref{thm:percolation}. Define $\widetilde Y_i$ to be $1$ if $Y_j=1$ for all $j\in N_i$ (recall that $N_i$ is the set of hexagons that intersect $Q_i$) and 
   the largest component of the region $R(\Phi)\cap (\cup_{j\in N_i}Q_j)$, which we denote by $X$, is such that $\cup_{j\in N_i}Q_j\setminus X$ contains only components of diameter smaller than $\delta\sqrt{t}/10$. 
   This means that if there exists a path $\widetilde Y_1,\widetilde Y_2,\ldots$ such that, for all $j=1,2,\ldots,$ we have $\widetilde Y_j=1$ and $Q_j$ and $Q_{j+1}$ intersect, then the region 
   $R(\Phi)\cap (Q_1\cup Q_2 \cup \cdots)$ has a connected component that intersects each hexagon $Q_j$ in the path. 
   Then, for $t$ sufficiently large, we can apply~\cite[Theorem~2]{PenPis} 
   and~\cite[Theorem~1.3]{LSS97} as before to show that the hexagons with $\widetilde Y_i=1$ percolate and, consequently, 
   adding balls of radius $r>r_\critical^{2/\sqrt{3} - \sqrt{\delta}}$ centered at 
   the nodes of $\Phi$ produces an infinite component. By the scaling argument in~\eqref{eq:rscaling} we 
   have that $r_\critical^{2/\sqrt{3}-\sqrt{\delta}}=r_\critical^{2/\sqrt{3}}\sqrt{\frac{\lambda}{\lambda-\sqrt{\delta}}}$, which finally yields
   $$
      \limsup_{t\to\infty} r_\critical(t) 
      \leq r_\critical^{2/\sqrt{3}-\sqrt{\delta}}
      = r_\critical^{2/\sqrt{3}}\sqrt{\frac{\lambda}{\lambda-\sqrt{\delta}}}.
   $$
   Since $\delta$ can be arbitrarily close to $0$, we obtain $\limsup_{t\to\infty}r_\critical(t)\leq r_\critical^{2/\sqrt{3}}$, which together with~\eqref{eq:liminf} concludes the proof of Theorem~\ref{thm:radius}.
\end{proof}

%############################################################################################
%############################################################################################
%############################################################################################
\section{Short time}\label{sec:smalltime}
Now we turn our attention to the case when $t$ is sufficiently small. 
We establish that, given a Monte Carlo estimate, $R(\Pi_t)$ contains an infinite component almost surely for sufficiently small $t$.

Consider a tessellation of $\mathbb{R}^2$ into regular hexagons of side length $50$. 
We will denote this tessellation by $\mathcal{H}_{50}$. Instead of 
considering the usual tessellation, where each hexagon is obtained by the union of some triangles of $\mathcal{T}$, we will shift the hexagonal tessellation
(see the illustration in Figure~\ref{fig:hextess}) 
so that no node of $\Pi_0$ is on an edge or vertex of $\mathcal{H}_{50}$, and the edges of $\mathcal{H}_{50}$ intersect as many of the balls centered at $\Pi_0$ as possible.
More formally,
since a transitive lattice can be specified by a single edge, we define $\mathcal{T}$ as the triangular lattice 
containing an edge between the points $(0,0)$ and $(1,0)$, and for any $\ell>0$, 
we let $\mathcal{H}_{\ell}$ be the hexagonal lattice containing an edge between
$(1/2,-\sqrt{3}/4)$ and $(\ell+1/2,-\sqrt{3}/4)$.

Let $H_1$ and $H_2$ be two hexagons of $\mathcal{H}_{50}$ that have one edge in common, and denote this edge by $e$. 
Starting from $e$, denote the other edges of $H_1$ in clockwise direction by $e_1,e_2,e_3,e_4,e_5$; thus $e_3$ is the edge of $H_1$ opposite to $e$. 
Similarly, denote the other edges of 
$H_2$ in clockwise direction by $e'_1,e'_2,e'_3,e'_4,e'_5$ (refer to Figure~\ref{fig:hextess}).
Given any three sets $X_1,X_2,X_3\subset \mathbb{R}^2$ and any $t>0$, 
we say that $R(\Pi_t)$ has a path from $X_1$ to $X_2$ inside $X_3$ if there exists a sequence of 
nodes $u_1,u_2,\ldots,u_k$ of $\Pi_t$, all of which inside $X_3$, such that $B(u_1,1/2)$ intersects $X_1$, $B(u_k,1/2)$ intersects $X_2$, 
and for each $i\geq 1$, 
the distance between $u_i$ and $u_{i+1}$ is at most 1.
With this, we say that $R(\Pi_t)$ \emph{crosses} $H_1$ and $H_2$ if the following three conditions hold:
\begin{enumerate}
   \item $R(\Pi_t)$ has a path from $e_3$ to $e_3'$ inside $H_1 \cup H_2$.
   \item $R(\Pi_t)$ has a path from $e_1\cup e_2$ to $e_4\cup e_5$ inside $H_1 \cup H_2$.
   \item $R(\Pi_t)$ has a path from $e_1'\cup e_2'$ to $e_4'\cup e_5'$ inside $H_1 \cup H_2$.
\end{enumerate}
We denote by $A_t$ the event that $R(\Pi_t)$ crosses $H_1$ and $H_2$ with a path that also crosses $H_1$ and $H_2$ at time $0$; this last condition is used to obtain a 
type of monotonicity later.
Then we have the following result.
\begin{theorem}\label{thm:smalltime}
   Suppose that there exists an $\epsilon_0>0$ such that $\pr{A_{\epsilon_0}}>0.8639$. Then, for all $\epsilon \in [0,\epsilon_0]$, it holds that 
   $R(\Pi_\epsilon)$ contains an infinite connected component almost surely.
\end{theorem}

\begin{figure}[tbp]
   \begin{center}
       \includegraphics[scale =.7]{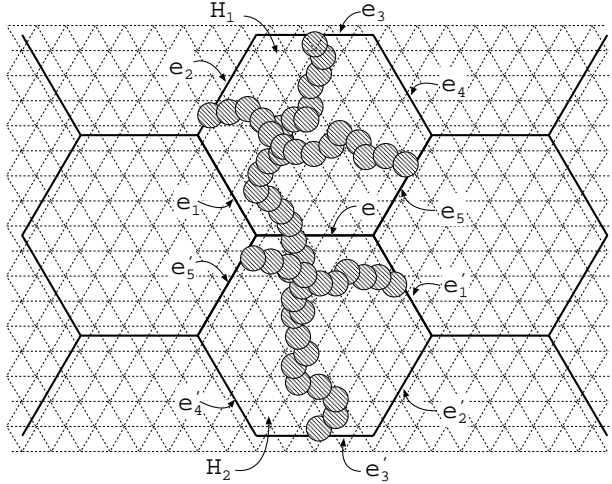}
   \end{center}
   \caption{The hexagonal tessellation $\mathcal{H}_4$ of $\mathbb{R}^2$ with hexagons of side length $4$, and the illustration of a path of intersecting balls crossing $H_1$ and $H_2$.}
%    \caption{(a) The hexagonal circle packing with balls centered at the vertices of the triangular lattice $\mathcal{T}$ (dotted lines). 
%    (b) A shifted hexagonal tessellation of $\mathbb{R}^2$ into hexagons of side length $4$, and the illustration of a path of intersecting balls crossing $H_1$ and $H_2$.
%    (c) The triangular lattice (solid lines) and its dual hexagonal lattice (dotted lines).}
   \label{fig:hextess}
\end{figure}

Note that, for any fixed $t$, verifying the condition $\pr{A_t}>0.8639$ resorts to solving a finite, but high dimensional integral describing the crossing probability. 
We were able to check the validity of this condition for $t=0.01$ via a Monte Carlo analysis\footnote{To obtain this Monte Carlo estimate we employed the Mersenne Twister pseudorandom number generator 
by Matsumoto and Nishimura~\cite{pseudo} with period $2^{19937}-1$ and improved initialization scheme from January $26^\text{th}$, 2002.} with confidence $99.99\%$.

We start the proof with the lemma below, which  establishes a type of monotonicity that will be useful later. 
To state the lemma, let $V$ be a set of points in $\mathbb{R}^d$ and define $E(V)$ as the set of pairs of points of $V$ whose distance is at most $1$. 
Note that the pair $(V,E(V))$ induces a graph over $V$.
\begin{lemma}\label{lem:monotone}
   Let $V\subseteq \Pi_0$ such that the graph $(V,E(V))$ is connected.
   Let $V_s$ be obtained by letting the nodes of $V$ move for time $s\geq 0$ according to independent Brownian motions. Then, 
   we have that $\pr{E(V) \subseteq E(V_s)}$ is non-increasing with $s$.
\end{lemma}
\begin{proof}
   This follows by Brownian scale. 
   Consider $s'>s$, and let $V'_{s'}$ be obtained by letting the nodes of $V$ move according 
   to independent Brownian motions for time $s'$.
   Now we create a coupling between $E(V'_{s'})$ and $E(V_s)$ so that,
   if $E(V) \subseteq E(V'_{s'})$, then $E(V)\subseteq E(V_s)$. 
   
   For each $u\in V$, let $\zeta_u(s)$ and $\zeta'_u(s')$ be the Brownian motions for the motion of $u$ in $V_s$ and $V_{s'}'$, respectively.
   Then,
   $$
      \pr{E(V) \subseteq E(V_s)}
      = \pr{\bigcap_{(u,v)\in E(V)}\Big\{\|u+\zeta_{u}(s)-v-\zeta_{v}(s)\|_2\leq 1\Big\}}.
   $$
   Now, by Brownian scale, we can couple $\zeta_u(s)$ and $\zeta'_{u}(s')$ via $\zeta_u(s) = \sqrt{s/s'}\zeta'_u(s')$. Using this, we write the 
   right-hand side above as 
   $$
      \pr{\bigcap_{(u,v)\in E(V)}\Big\{\|u-v+\sqrt{s/s'}(\zeta'_u(s')-\zeta'_v(s'))\|_2\leq 1\Big\}}.
   $$
   Now define the vectors $x_1=u-v$ and $x_2=\zeta'_{u}(s')-\zeta'_{v}(s')$.
   For any $\delta\in(0,1)$, it follows by standard geometric arguments that, if $\|x_1\|_2$ and $\|x_1+x_2\|_2$ are at most $1$, then
   $$
      \|x_1 + \delta x_2 \|_2 \leq 1,
   $$
   which establishes that 
   \begin{eqnarray*}
      \lefteqn{\pr{\bigcap_{(u,v)\in E(V)}\Big\{\|u-v+\sqrt{s/s'}(\zeta'_u(s')-\zeta'_v(s'))\|_2\leq 1\Big\}}}\\
      &\geq& \pr{\bigcap_{(u,v)\in E(V)}\Big\{\|(u-v)+(\zeta'_{u}(s')-\zeta'_{v}(s'))\|_2\leq 1\Big\}}= \pr{E(V) \subseteq E(V'_{s'})},
   \end{eqnarray*}
   which completes the proof.   
\end{proof}

Now we proceed to the proof of Theorem~\ref{thm:smalltime}.
\begin{proof}[{\bf Proof of Theorem~\ref{thm:smalltime}}]
   Note that, due to Lemma~\ref{lem:monotone}, proving Theorem~\ref{thm:smalltime} reduces to showing that, if there exists an $\epsilon>0$ such that 
   $\pr{A_\epsilon}>0.8639$, then $R(\Pi_\epsilon)$ contains an infinite connected component almost surely. We henceforth fix a value of 
   $\epsilon$ and assume that $\pr{A_\epsilon}>0.8639$.   

   We will use a renormalization argument. 
   Consider the hexagons $\mathcal{H}_{50}$ described in the beginning of this section.
   Now, define the graph $L=(U,F)$ such that 
   $U$ is the set of points given by the centers of the hexagons
   and $F$ is the set of edges between every pair of points $i,j\in U$ for which the hexagons with centers at $i$ and $j$ share an edge. 
   Note that $L$ consists of a scaling of the triangular lattice. 
   
   We now define a collection of random variables $X_i$ for each edge $i\in F$.
   In order to explain the process defining $X_i$, let $H_1$ and $H_2$ be the hexagons whose centers are the endpoints of $i$. 
   We then define $X_i=1$ if and only if $R(\Pi_\epsilon)$ crosses $H_1$ and $H_2$ with a path of balls that also crosses $H_1$ and $H_2$ at time $0$.
   (The definition of crossings is given right before the statement of Theorem~\ref{thm:smalltime}.)
   %Note that, at time $\epsilon>0$, we have that no node of $\Pi_\epsilon$ is located on an edge or vertex of the hexagonal tessellation. 
   Let $j$ be an edge such that $i$ and $j$ are disjoint, and let $H_3$ and $H_4$ denote the hexagons centered at the endpoints of $j$.
   Clearly, $X_i$ and $X_j$ are independent since the set of balls crossing $H_1$ and $H_2$ at time $0$ 
   does not intersect the set of balls crossing $H_3$ and $H_4$ at time $0$. 
   Thus, the collection $(X_i)_i$ is a so-called 1-dependent bond percolation process, with $\pr{X_i=1}=\pr{A_\epsilon}>0.8639$.
   Then, we can use a result of Balister, Bollob\'as and Walters~\cite[Theorem~2]{BBW05}, which gives that any 1-dependent bond percolation process on the square lattice with marginal probability
   larger than 0.8639 percolates almost surely. Since the triangular lattice contains the square lattice, we obtain that, almost surely,
   there exists an infinite path of consecutive edges of $F$ with $X_i=1$ for all $i$ in the path.
   
   To conclude the proof, note that, 
   for two non-disjoint edges $i$ and $j$ with $X_i,X_j=1$, we have that the crossings of the hexagons whose centers are located at the endpoints of $i$ and $j$
   intersect. Then, the infinite path of $X_i$ with $X_i=1$ for all $i$ in the path contains an infinite path inside $R(\Pi_\epsilon)$, which concludes the proof of Theorem~\ref{thm:smalltime}.
\end{proof}

%############################################################################################
%############################################################################################
%############################################################################################
\section{Extensions and Open Problems}\label{sec:open}
In the remaining of this section we discuss extensions and open problems regarding other circle packings (Section~\ref{sec:packings}), 
balls moving over graphs (Section~\ref{sec:graphs}) and critical radius for non-mobile point processes (Section~\ref{sec:radius})

\subsection{Other circle packings}\label{sec:packings}
Let $\Pi_0^\mathrm{s}$ be 
the point process given by the vertices of the square lattice with side length $1$, and let $\Pi_t^\mathrm{s}$ be the point process obtained by letting 
the nodes of $\Pi_0^\mathrm{s}$ move for time $t$ according to independent Brownian motions. 
Note that, for any $\epsilon>0$, if we look at two balls of radius $1/2$ centered 
at two adjacent nodes of $\Pi_0^\mathrm{s}$, then at time $\epsilon$, the probability that these two balls intersect is strictly smaller than $1/2$, which is the 
critical probability for edge percolation on the square lattice~\cite{Grimmett}. This motivates our next conjecture.
\begin{conjecture}
   For any $\epsilon>0$, it holds that, almost surely, all components of $R(\Pi_\epsilon^\mathrm{s})$ are finite.
\end{conjecture}

Now we consider the question of whether percolation is a \emph{monotone} property.
We say that a point process $\Pi_0$ is transitive if, for every two nodes $v,v'\in \Pi_0$, there 
exists an isometry $f\colon \Pi_0\to\Pi_0$ such that $f(v)=v'$. The open problem below concerns the question
of whether transitivity is enough to obtain monotonicity in the percolation properties of balls moving as Brownian motion.
\begin{question}\label{con:monotone}
   Let $\Pi_0$ be a transitive point process so that $R(\Pi_0)$ is a connected set. 
   Let $\Pi_t$ be obtained from $\Pi_0$ by letting the nodes move as independent Brownian motions for time $t$.
   Then, if for some time $t_0$ we have that $R(\Pi_{t_0})$ has an infinite component almost surely, then, is it true that, for any $t<t_0$, 
   $R(\Pi_t)$ also has an infinite component almost surely? 
   Similarly, if for some $t_1$ we have that $R(\Pi_{t_1})$ contains only finite components almost surely, 
   then, does it hold that, for any $t>t_1$, $R(\Pi_t)$ also contains only finite components almost surely?
\end{question}
\begin{remark}
\rm{We note that Question~\ref{con:monotone} above is false if we drop the condition that $\Pi_0$ is transitive. 
For example, consider a tessellation of $\mathbb{R}^2$ into squares of side length $6$ and, in each square of the tessellation, consider the 
configuration of balls illustrated in Figure~\ref{fig:hexmonotone}, where each ball has radius $1/2$, solid balls represent the superposition of 
14 balls and white balls represent a single ball. It is easy to see that, at 
a sufficiently small time $\epsilon$, the union of the balls will not contain an infinite component almost surely. However, the density of 
balls is equal to $\frac{9\times 14 + 18}{36}=4$ and, as the balls move for a sufficiently large amount of time, their position will approach a Poisson point process which is 
known to percolate.
}
\end{remark}
\begin{figure}[tbp]
   \begin{center}
      \includegraphics[scale =.6]{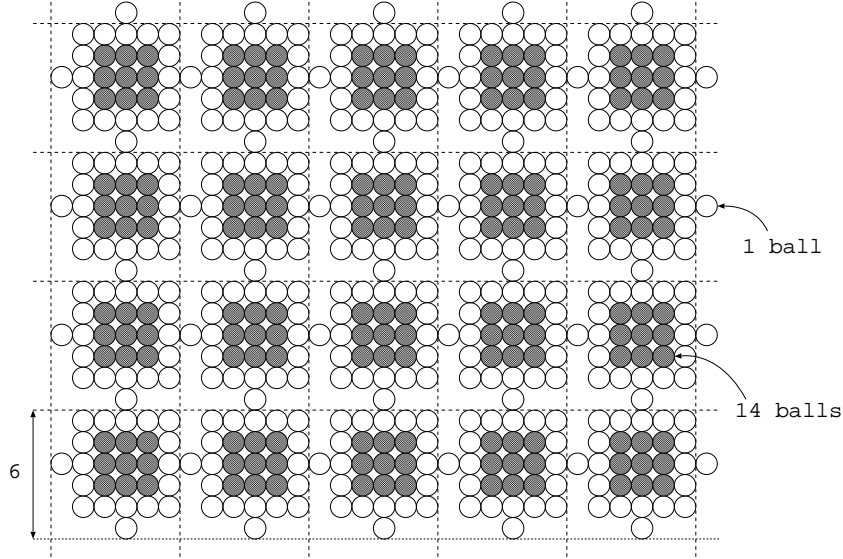}
   \end{center}
   \caption{An example of a non-transitive configuration of balls that is not monotone. Each solid ball represents the superposition of $14$ balls and 
   white balls represent single balls.}
   \label{fig:hexmonotone}
\end{figure}

\subsection{Motion over graphs}\label{sec:graphs}
We now consider the case when the motion of the nodes is more restricted. First, let $\Pi_0$ be the point process given by the integer points 
of $\mathbb{R}$. For any node $u\in\Pi_0$, we let $u+\zeta_u(t)$ be its position at time $t$, where $(\zeta_u(t))_t$ is a one-dimensional Brownian motion.
Now, consider a sequence of $m$ distinct nodes 
$u_1,u_2,\ldots,u_m$ such that $B(u_i,1/2)$ and $B(u_{i+1},1/2)$ intersect for all $i$. We call such a sequence of nodes as a \emph{path}.
Let $\epsilon$ be a sufficiently small positive constant, and 
consider only the nodes of $\Pi_0$ whose displacement from time $0$ to time $\epsilon$ is smaller than $1/2$; we denote these nodes as \emph{good} nodes. We claim that 
\begin{equation}
   \prcond{\text{$u_1,u_2,\ldots,u_m$ form a path at time $\epsilon$}}{\text{$u_i$ is good for all $i$}} = \frac{1}{m!}.
   \label{eq:observ}
\end{equation}
In order to see this, suppose, without loss of generality, that $u_1<u_2<\cdots<u_m$. 
For each node $u\in \Pi_0$, let $\zeta_u'(\epsilon)$ be the displacement of $u$ from time $0$ to $\epsilon$ given that $u$ is a good node.
Then, in order for $B(u_1+\zeta_{u_1}(\epsilon),1/2)$ to intersect 
$B(u_2+\zeta_{u_2}(\epsilon),1/2)$ we need that $|u_1+\zeta'_{u_1}(\epsilon)-u_2-\zeta'_{u_2}(\epsilon)|\leq 1$. Since $u_1$ and $u_2$ are good nodes, 
this condition translates to $u_2+\zeta'_{u_2}(\epsilon)-u_1-\zeta'_{u_1}(\epsilon)\leq 1$, which in turn implies that 
$\zeta'_{u_1}(\epsilon)\geq\zeta'_{u_2}(\epsilon)$. Repeating this argument, we obtain the condition
$\zeta'_{u_1}(\epsilon)\geq\zeta'_{u_2}(\epsilon)\geq\zeta'_{u_3}(\epsilon)\geq\cdots\geq\zeta'_{u_m}(\epsilon)$. Since the $\zeta'$ are independent and identically 
distributed, we have that $\pr{\zeta'_{u_1}(\epsilon)\geq\zeta'_{u_2}(\epsilon)\geq\cdots\geq\zeta'_{u_m}(\epsilon)}=1/m!$,
which establishes~\eqref{eq:observ}. 

We now consider a more general scenario. 
Let $G$ be an infinite graph that is vertex transitive and has bounded degree. 
We assume that each edge of $G$ has length $1$, which gives a metric over $G$.
Let $\Pi_0(G)$ be the point process given by putting one node at each vertex of $G$ and define 
$\Pi_t(G)$ as the point process obtained by letting the nodes of $\Pi_0(G)$ move for time $t$ along the edges of $G$ according to independent Brownian motions. 
Then $R(\Pi_t(G))$ is the union of balls centered at the nodes of $\Pi_t$ and having radius $1/2$ with respect to the metric induced by $G$.
We note that the probability given in~\eqref{eq:observ} for any fixed path $u_1,u_2,\ldots, u_m$ of good nodes to form a path at a time 
$\epsilon$ that is sufficiently small is at most $1/m!$. This motivates our next conjecture.
\begin{conjecture}
   Let $G$ be an infinite graph that is vertex transitive and has bounded degree. 
   Then, for any $t>0$, the region $R(\Pi_t(G))$ contains only finite components almost surely.
\end{conjecture}

\subsection{Critical radius of point processes}\label{sec:radius}
Here we let $\Pi$ be a point process over $\mathbb{R}^2$ and consider the region $R(\Pi,r)$ as the union of balls of radius $r$ centered at the nodes of 
$\Pi$. In this section, we only consider point processes with unit intensity and let $r_\critical(\Pi)$ be the smallest $r$ for which $R(\Pi,r)$ contains an infinite 
component. It is intuitive to believe that point processes that are more organized have smaller critical radius; this is the core of our next conjecture.
For more information on zeros of Gaussian analytic functions, we refer to~\cite{HKPV}.
\begin{conjecture}
   Let $\Pi_\mathrm{L}$ be any transitive point process with intensity $1$ (as defined before Question~\ref{con:monotone}).
   Let $\Pi_\mathrm{GAF}$ be a point process given by the zeros of a Gaussian analytic function with intensity 1
   and $\Pi_\mathrm{P}$ be a Poisson point process with intensity $1$. Then, 
   $$
      r_\critical(\Pi_\mathrm{L}) < r_\critical(\Pi_\mathrm{GAF}) < r_\critical(\Pi_\mathrm{P}).
   $$
\end{conjecture}

Finally, consider a Poisson point process $\Pi$ with intensity $1$ over $\mathrm{R}^d$ and let $r_\critical$ be the critical radius for percolation of 
balls centered at the nodes of $\Pi$. Our last open problem concerns small perturbations of the critical radius.
\begin{question}
   Let $\epsilon>0$ and, for each node $v\in \Pi$, let $X_v$ be a uniform random
   variable over $[-\epsilon,\epsilon]$. 
   For each node $v\in\Pi$, add a ball of radius $r_\critical+X_v$ centered at $v$. Will the union of the balls
   contain an infinite component almost surely?
\end{question}

%############################################################################################
%############################################################################################
%############################################################################################
\bibliographystyle{plain}
\bibliography{hexagonal}

\begin{thebibliography}{10}

\bibitem{AlonSpencer}
N.~Alon and J.H. Spencer.
\newblock {\em The probabilistic method}.
\newblock John Wiley \& Sons, 3rd edition, 2008.

\bibitem{BBW05}
P.~Balister, B.~Bollob\'as, and M.~Walters.
\newblock Continuum percolation with steps in the square or the disc.
\newblock {\em Random structures and algorithms}, 26:392--403, 2005.

\bibitem{BR}
B.~Bollob\'as and O.~Riordan.
\newblock {\em Percolation}.
\newblock Cambridge University Press, 2006.

\bibitem{Conway}
J.H. Conway and N.J.A. Sloane.
\newblock {\em Sphere packings, lattices and groups}.
\newblock Springer-Verlag, 3rd edition, 1999.

\bibitem{Grimmett}
G.~Grimmett.
\newblock {\em Percolation}.
\newblock Springer-Verlag, 2nd edition, 1999.

\bibitem{HKPV}
J.B. Hough, M.~Krishnapur, Y.~Peres, and B.~Vir\'ag.
\newblock {\em Zeros of {G}aussian Analytic functions and determinantal point
  processes}.
\newblock American Mathematical Society, 2009.

\bibitem{LSS97}
T.M. Liggett, R.H. Schonmann, and A.M. Stacey.
\newblock Domination by product measures.
\newblock {\em The Annals of Probability}, 25:71--95, 1997.

\bibitem{pseudo}
M.~Matsumoto and T.~Nishimura.
\newblock {M}ersenne twister: A 623-dimensionally equidistributed uniform
  pseudo-random number generator.
\newblock {\em {ACM} Transactions on Modeling and Computer Simulation},
  8:3--30, 1998.

\bibitem{MR}
R.~Meester and R.~Roy.
\newblock {\em Continuum Percolation}.
\newblock Cambridge University Press, 1996.

\bibitem{BM}
P.~{M\"orters} and Y.~Peres.
\newblock {\em Brownian Motion}.
\newblock Cambridge University Press, 2010.

\bibitem{PenPis}
M.~Penrose and A.~Pisztora.
\newblock Large deviations for discrete and continuous percolation.
\newblock {\em Advances in Applied Probability}, 28:29--52, 1996.

\bibitem{PSSS11}
Y.~Peres, A.~Sinclair, P.~Sousi, and A.~Stauffer.
\newblock Mobile geometric graphs: detection, coverage and percolation.
\newblock In {\em Proceedings of the 22st {ACM-SIAM} Symposium on Discrete
  Algorithms ({SODA})}, pages 412--428, 2011.

\bibitem{SS10}
A.~Sinclair and A.~Stauffer.
\newblock Mobile geometric graphs, and detection and communication problems in
  mobile wireless networks, 2010.
\newblock {arXiv:1005.1117v2}.

\bibitem{S11}
A.~Stauffer.
\newblock Adversarial detection and space-time percolation in mobile geometric
  graphs, 2011.
\newblock {arXiv:1108.6322v1}.

\end{thebibliography}

%############################################################################################
%############################################################################################
%############################################################################################
\appendix

%############################################################################################
%############################################################################################
%############################################################################################
\section{Standard large deviation results}

We use the following standard Chernoff bounds during our proofs.

\begin{lemma}[Chernoff bound for Poisson]\label{lem:cbpoisson}
   Let $P$ be a Poisson random variable with mean $\lambda$. Then, for any
   $0<\epsilon<1$,
   $$
      \pr{P \geq (1+\epsilon) \lambda} \leq  \exp\left(-\frac{\lambda \epsilon^2}{2}(1-\epsilon/3)\right)
   \quad\text{ and }\quad
      \pr{P \leq (1-\epsilon) \lambda} \leq  \exp\left(-\frac{\lambda \epsilon^2}{2}\right).
   $$
\end{lemma}

\begin{lemma}[Chernoff bound for Binomial, see~{\cite[Lemma~A.1.4]{AlonSpencer}}]\label{lem:cbbinomial}
   Let $X_1,X_2,\ldots,X_n$ be independent Bernoulli random variable such that $\E{X_i}=p_i$. Let $X=\sum_{i=1}^n X_i$. Then, for any 
   $\epsilon>0$,
   $$
      \pr{X \geq (1+\epsilon) \E{X}} \leq  \exp\left(-\frac{2\epsilon^2(\E{X})^2}{n}\right)
   \quad\text{ and }\quad
      \pr{P \leq (1-\epsilon) \lambda} \leq  \exp\left(-\frac{\lambda \epsilon^2}{2}\right).
   $$
\end{lemma}

\begin{lemma}[Gaussian tail bound~{\cite[Theorem~12.9]{BM}}]\label{lem:gaussiantail}
   Let $X$ be a normal random variable with mean $0$ and variance $\sigma^2$. Then, for any $R \geq \sigma$ we have that
   $\pr{X \geq R} \leq \frac{\sigma}{\sqrt{2\pi}R}\exp\left(-\frac{R^2}{2\sigma^2}\right)$.
\end{lemma}

\end{document}